\documentclass[a4paper,12pt]{article}

\usepackage{pstricks}
\usepackage{graphicx,psfrag}
\usepackage{amsmath}
\usepackage{amssymb}
\usepackage{amsthm}
\usepackage{color}

                \newcommand {\bt}  {\beta}
                
                \newcommand {\Del} {\Delta}
              \newcommand {\ve}   {\varepsilon}

      \newcommand {\HHH}    {\mathbb{H}}       
      \newcommand {\RRR}  {{\mathbb R}}

        \newcommand {\CCC}  {\mathbb{C}}     \newcommand {\zzz}  {\zeta}  
                                           
     \newcommand {\beq}  {\begin{equation}}  \newcommand {\beqo}  {\begin{equation*}}
      \newcommand {\eeq}  {\end{equation}}   \newcommand {\eeqo}  {\end{equation*}}

     \newcommand {\Hb} {\HHH_{\rm b}}    \newcommand {\Hs} {\HHH_{\rm s}}
     \newcommand {\Hbs} {\HHH_{\rm b\rightarrow s}}    \newcommand {\Hsb} {\HHH_{\rm s\to b}}
\newcommand {\Hg} {\HHH_{\rm g}}

      \newtheorem{theorem}{Theorem}
      \newtheorem{lemma}{Lemma}
      
      \newtheorem{zam}{Remark}
      \newtheorem{opr}{Definition}
      \newtheorem{cor}{Corollary}

\title{Rotating rod and ball}

\author{Sergey Kryzhevich \thanks{Institute of Applied Mathematics, Faculty of Applied Physics and Mathematics, Gda\'nsk University of Technology, Gda\'nsk, Poland and BioTechMed Center, Gda\'nsk University of Technology, Gda\'nsk, Poland, serkryzh@pg.edu.pl} \and Alexander Plakhov\thanks{Center for R{\&}D in Mathematics and Applications, Department of Mathematics, University of Aveiro, Portugal and Institute for Information Transmission Problems, Moscow, Russia, plakhov@ua.pt}}

\begin{document}
\maketitle


\begin{abstract}
We consider a mechanical system consisting of an infinite rod (a straight line) and a ball (a massless point) on the plane. The rod rotates uniformly around one of its points. The ball is reflected elastically when colliding with the rod and moves freely between consecutive hits. A sliding motion along the rod is also allowed. We prove the existence and uniqueness of the motion with a given position and velocity at a certain time instant. We prove that only 5 kinds of motion are possible: a billiard motion; a sliding motion; a billiard motion followed by sliding; a sliding motion followed by a billiard one; and a constant motion when the ball is at the center of rotation. The asymptotic behaviors of time intervals between consecutive hits and of distances between the points of hits on the rod are determined.
\end{abstract}

\begin{quote}
\end{quote}

\begin{quote}
\end{quote}

\section{Introduction}
In this paper, we study the billiard in a rotating half-plane.

There has been a considerable amount of work concerning time-dependent billiards. The main motivation for this work was concerned with studying mathematical models of Fermi acceleration.
To a large extent, the research was focused on the question, of whever the energy of the billiard ball can grow to infinity as a result of multiple collisions with moving boundaries. Some papers studied billiard in oscillating domains \cite{BreathingCircle,Car,Dol,Gel2008,Gel1,Gel2,INei, Los}, while others \cite{rotSquare,Burdzy, rotOval, rot1,rot2} were concerned with billiard in rotating domains. This list of course is far from being complete.

The motivation for our study, however, comes from a different physical model going back to Isaac Newton. In the second book of his {\it Principia} \cite{N} he studied the problem of least resistance for a body moving in a medium composed of point particles. The medium is extremely rare, so as mutual interaction of particles can be neglected, and collisions of the particles with the body are perfectly elastic.
 Newton considered the problem in the class of convex axisymmetric bodies with fixed length and width. 

Starting from the early 1990s, many papers concerning the problem of minimal resistance in various classes of bodies have appeared (see, e.g., \cite{BrFK, BFK,BK,BG97, LO,LP1,canad,PT,SIREV}, among others). It is generally assumed in these papers that the body translates in the medium; see, however, the papers \cite{Kr,ARMA,PTG,OMT}, where a combination of translational and rotational motions is considered.

To the best of our knowledge, a regular study concerning the free motion of a body (involving both translation and rotation) in the framework of Newtonian aerodynamics has never been carried out, even in the 2D case. Theorems of existence and uniqueness for the dynamics have not been obtained, and free motion on the plane of special shapes, even the simplest ones, such as an ellipse, a square, or a line segment, has never been studied.\footnote{The only exception is the disk, whose dynamics is trivial: its motion is rectilinear, and its scalar velocity satisfies a differential equation $\dot v = -c v^2$.}

Here we concentrate on the case which seems to be the simplest one: a line segment. By simplifying further the problem, we assume that the ball and the length of the segment are infinite. Since the medium particles do not mutually interact, it suffices to consider the interaction of a single particle with the segment. Thus, we come to a mechanical system composed of a ball (a mass point) and a rod (a straight line) rotating uniformly about a fixed point of the rod.  When colliding with the rod, the ball is reflected elastically. The problem is to describe the dynamics of the ball.

Thus, we have the billiard in a half-plane rotating uniformly about a fixed point on its boundary. Apparently, such a system has never been studied. It looks quite simple, but its study is far from being trivial, as will be seen in this paper later on.

\begin{figure}[!ht]
\begin{center}
\includegraphics[height=2.5in]{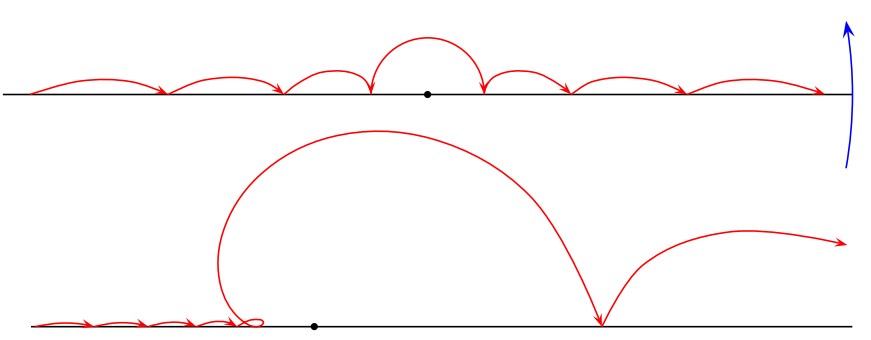}
\caption{Two trajectories of the ball in the rotating coordinate system. The rod rotates counterclockwise. In the second trajectory, there is a grazing impact.}
\end{center}
\end{figure}

\begin{figure}[!ht]
\begin{center}
\includegraphics[height=2.5in]{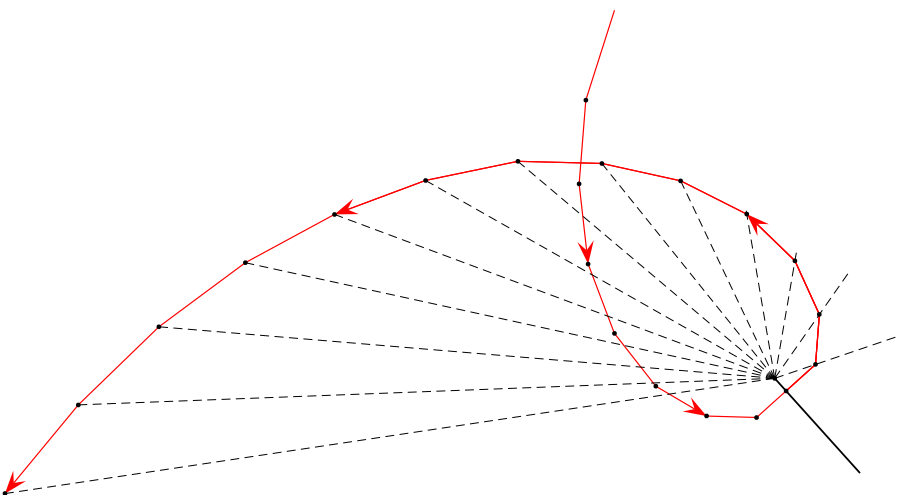}
\caption{A trajectory of the ball in a resting coordinate system.}
\end{center}
\end{figure}

This article is a continuation of an earlier work published in \cite{KrPl1}.

Without loss of generality assume that the angular velocity of the rod equals 1 and the rotation is counterclockwise. It is convenient to consider the dynamics in the rotating coordinate system on the complex plane $\mathbb{C}$ (see Figures 1 and 2), where the rod is represented by the real axis $\RRR$, with the fixed point being at the origin, and the position of the ball $\zzz(t)$ at the time instant $t \in \mathbb{R}$ is always contained in the closed upper half-plane $\mathbb{C}_0^+ := \{ z \in \mathbb{C} : \text{Im}\, z \ge 0 \}$.

We also denote the open upper half-plane as $\mathbb{C}^+ := \{ z \in \mathbb{C} : \text{Im}\, z > 0 \}$.

According to Newton's second law, the acceleration of the ball times its mass is equal to the force acting on it. Direct calculation shows that the acceleration (that is, the second derivative of the ball's position) in the rotating coordinate system equals $\ddot\zzz + 2i\dot\zzz - \zzz$, where dot means the derivative with respect to $t$. It is natural to consider this function as a distribution since the force becomes infinite at the instants of hits. The force equals zero, when the ball is located outside the rod, is orthogonal to the rod, and points inward to the upper half-plane, when the ball is on the rod. In the rotating coordinate system, this can be written down as follows,
\beq\label{force}
\ddot\zzz + 2i\dot\zzz - \zzz\
\left\{ \begin{array}{ll} = 0, & \text{if} \  \zzz \in \CCC^+;\\ \in i\RRR^+_0, & \text{if} \  \zzz \in \RRR.
\end{array} \right.
\eeq
Equation \eqref{force} should be understood in terms of distributions. The function $\zzz$ is continuous.
The first line in this equation means that if a test function (a compactly supported $C^\infty$ function of one variable) $\phi$ is supported in the open set $\{ t :\, \zzz(t) \in \CCC^+ \}$, then  the value of the functional $\ddot\zzz + 2i\dot\zzz - \zzz$ at $\phi$ equals zero,
$\langle \ddot\zzz + 2i\dot\zzz - \zzz\,,\, \phi \rangle = \int_{\RRR} (\ddot\phi - 2i\dot\phi - \phi)\, \zzz\, dt = 0.$
The second line means that if the test function is nonnegative, $\phi \ge 0$, then
$ \int_{\RRR} (\ddot\phi - 2i\dot\phi - \phi)\, \zzz\, dt \in i\RRR_0^+.$


We additionally assume that the ball reflects elastically when hitting the real line, that is,
\beq\label{one-sided}
\text{ if \, $\zzz(t) \in \RRR$ \, then} \quad \dot\zzz(t+0) = \dot\zzz(t-0)^*;
\eeq
here asterisk means the complex conjugation.
If $\dot\zzz(t-0) \in \RRR$ in \eqref{one-sided}, and therefore, the function $\zzz$ is differentiable at $t$,
$\dot\zzz(t) = \dot\zzz(t+0) = \dot\zzz(t-0)$, and $\zzz \in \CCC^+$ in a punctured neighborhood of $t$, then we say that a {\it grazing impact} takes place at the instant $t$.

Taking account of \eqref{one-sided}, it is natural to introduce an equivalence relation as follows: if $x \in \RRR$ then $(x, z) \sim (x, z^*)$. In other words, $(z_1, z_2)$ and $(w_1, w_2)$ are equivalent, {\it iff} $w_1 = z_1 \in \RRR$ and $w_2 = z_2^*$ (we identify the coordinates of the ball immediately before the impact and immediately after it).
Denote $\HHH := (\CCC^+_0 \times \CCC)/\sim$.

\begin{opr}
Consider a continuous function $\zzz(t),\, t \in \RRR$ taking values in $\CCC_0^+$ such that

(a) equation \eqref{force} is satisfied in the sense of distributions;

(b)  
 if $\zzz(t) \in \RRR$, there exist the one-sided derivatives of $\zzz$ satisfying $\dot\zzz(t+0) = \dot\zzz(t-0)^*.$\\
Slightly abusing the notation, the class of equivalence $\big\{  (\zzz(t),\, \dot\zzz(t-0)),\ (\zzz(t),\, \dot\zzz(t+0))  \big \} \in \HHH$ is denoted as $(\zzz(t),\, \dot\zzz(t))$.

The map $t \mapsto (\zzz(t),\, \dot\zzz(t))$ from $\RRR$ to $\HHH$ is called a {\rm motion}.

Likewise, a map defined on a (finite or semi-infinite) interval $I \subset \RRR$ satisfying (a) and (b) will also be called a motion.
If the set $\{ t \in I:\, \zzz(t) \in \RRR \}$ is locally finite, the motion is called a {\rm billiard motion}, and if $\zzz(I) \subset \RRR$, it is called a {\rm sliding motion}.
\end{opr}

\begin{zam}\label{z invert}
Note that if $(\zzz(t), \dot\zzz(t))$ is a motion, then the function obtained by time reversion and symmetries with respect to the imaginary and real axes, $(-\zzz(-t)^*, \dot\zzz(-t)^*)$, is also a motion.
\end{zam}

In the case of billiard motion there is a locally finite sequence of instants of hits $\ldots < t_{n-1} < t_n < t_{n+1} < \ldots$, so as $\zzz(t_n) \in \RRR$, and $\zzz(t)$ lies in $\CCC^+$ and satisfies $\ddot\zzz + 2i\dot\zzz - \zzz\ = 0$ for $t \in (t_n,\, t_{n+1})$. Solving this equation, one concludes that the position of the ball between consecutive hits is described by
\begin{equation}\label{e1}
\zzz(t) = (z + vt) e^{-it} \in \CCC^+, \qquad \text{where} \ \ z,\, v \in \CCC.
\end{equation}

In the case of sliding the ball moves along the rod, that is, $\zzz \in \RRR$, and $\zzz$ satisfies the equations
\beq\label{e2}
\ddot\zzz - \zzz = 0, \qquad \dot\zzz \ge 0.
\eeq
Solving them, one obtains all the possible kinds of sliding motion on intervals: (i) $\zzz(t) = k\sinh t$,\, $t \in \RRR$; (ii) $\zzz(t) = ke^t$ or $-ke^{-t}$,\, $t \in \RRR$; (iii) $\zzz(t) = k\cosh t$,\, $t \ge 0$ or $-k\cosh t$,\, $t \le 0$, where $k > 0$ is a constant; (iv) $\zzz(t) = 0$,\, $t \in \RRR$, and functions obtained from these ones by time shifts.

As will be seen below, all motions are either billiard or sliding ones or their combinations: a billiard motion followed by sliding or vice versa (mixed motions). Additionally, sliding and mixed motions are extremely rare events: the corresponding trajectories occupy a set of codimension 2 in the phase space $\HHH$.

Denote
$$
\Hb := (\CCC^+ \times \CCC) \cup \big( (\RRR \times (\CCC \setminus \RRR)) /\sim \big) \cup (\RRR \times \RRR^-),
$$ $$
\Hs := \RRR \times \RRR^+, \quad \Hbs :=\RRR^+ \times \{ 0 \}, \quad \Hsb := \RRR^- \times \{ 0 \}, \quad \Hg := \RRR \times \RRR^- \subset \Hb.
$$
The subscript "b" means billiard, "s" means sliding, "g" means grazing, and "$b\to s$",\, "$s \to b$" means switching from billiard motion to sliding one and vice versa. The entire phase space can be represented as the disjoint union
$$
\HHH = \Hb \cup \Hs \cup \Hbs \cup \Hsb \cup \{ (0,0) \}.
$$
This notation is explained by the following remark and lemma. 

\begin{zam} In other words, we claim that the considered system exhibits three types of motion.
\begin{enumerate}
\item The fixed point where the ball stands at the fixed point of the rod (point $(0,0)$).
\item The billiard motion of the ball with instantaneous collisions with the rod. This motion corresponds to the domain $\Hb$ and includes a) free motion  with no interaction with the rod $(\CCC^+ \times \CCC)$; b) impacts with non-zero normal velocity $(\RRR \times (\CCC \setminus \RRR))$; c) grazing where the normal velocity at the impact is zero but the impact is isolated in time $(\RRR \times \RRR^-)$.   
\item Sliding regime where the ball moves along the rod and does not leave it $(\Hs)$.
\end{enumerate}
Transitions between the two latter regimes are possible and happen from billiard to sliding mode at the points of the set $\Hbs$ and vice versa at points of the set $\Hsb$.
\end{zam}

\begin{lemma}\label{l notation}
For any $(z, \check{z}) \in \HHH$ there exist $\ve > 0$ and a motion $(\zzz(t), \dot\zzz(t)),\, t \in (-\ve,\, \ve)$ satisfying $(\zzz(0), \dot\zzz(0)) = (z, \check{z})$ such that

(a) if $(z, \check{z}) \in \Hb$, then $\zzz(t) \in \CCC^+$ for $t \in (-\ve,\, 0) \cup (0,\, \ve)$ (billiard motion);

(b) if $(z, \check{z}) \in \Hs \cup (0,0)$, then $\zzz(t) \in \RRR$ for $t \in (-\ve,\, 0) \cup (0,\, \ve)$ (sliding);

(c) if $(z, \check{z}) \in \Hbs$, then $\zzz(t) \in \CCC^+$ for $t \in (-\ve,\, 0)$ and $\zzz(t) \in \RRR$ for $t \in (0,\, \ve)$ (transi-\\ tion from billiard motion to sliding);

(d) if $(z, \check{z}) \in \Hsb$, then $\zzz(t) \in \RRR$ for $t \in (-\ve,\, 0)$ and $\zzz(t) \in \CCC^+$ for $t \in (0,\, \ve)$ (transi-\\ tion from sliding to billiard motion).
\end{lemma}

\begin{proof}
Let $(z, \check{z}) \in \CCC^+ \times \CCC$; then we are looking for the motion in the form \eqref{e1}. One easily checks that the function $\zzz(t) = [ z + (iz + \check{z})t ] e^{-it}$ satisfies the initial conditions, and therefore, is a solution, and $\zzz(t) \in \CCC^+$ for $t \in (-\ve,\, \ve)$ if $\ve > 0$ is sufficiently small.

Let $(z, \check{z}) = \big(\, (r, x-iy), \, (r, x+iy) \,\big)\in \big( (\RRR \times (\CCC \setminus \RRR)) /\sim \big)$,\, $r \in \RRR,\, y > 0$. We look for the solutions in the form \eqref{e1} for $t < 0$ and $t > 0$ separately, with the initial conditions $\zzz(0) = r,\, \dot\zzz(0^-) = x-iy,\, \dot\zzz(0^+) = x+iy$. One obtains $\zzz(t) = [r + xt + i(r - y)t ] e^{-it}$ for $t < 0$ and $\zzz(t) = [r + xt + i(r + y)t ] e^{-it}$ for $t > 0$. The left and right derivatives of Im$\,\zzz$ at 0 are\, Im$\,\dot\zzz(0^-) = -y < 0$ and Im$\,\dot\zzz(0^+) = y > 0$, hence $\zzz(t) \in \CCC^+$ for $t \in (-\ve,\, 0) \cup  (0,\, \ve)$, provided that $\ve > 0$ is sufficiently small.

Let $(z, \check{z}) = (r, x) \in \RRR \times \RRR^-$. Again, looking for the solution in the form \eqref{e1}, we obtain $\zzz(t) = (r + xt + irt) e^{-it}$. Since Im$\,\zzz(0) = 0$,\, Im$\,\dot\zzz(0) = 0$, and Im$\,\ddot\zzz(0) = -2x > 0$, we have $\zzz(t) \in \CCC^+$ for $t \in (-\ve,\, 0) \cup  (0,\, \ve)$ with $\ve > 0$ sufficiently small.

Let $(z, \check{z}) = (r, x) \in \Hs = \RRR \times \RRR^+$. Here $\zzz$ should satisfy equations \eqref{e2}.  Consider several cases.

(i) $x > |r|$. We are looking for a solution in the form $\zzz(t) = k\sinh(t+t_*)$. The initial conditions then take the form $k\sinh t_* = r,\, k\cosh t_* = x$ or, equivalently, $ke^{t^*} = \dfrac{x+r}{2}$,\, $ke^{-t^*} = \dfrac{x-r}{2}$,\, $k > 0$. These equations have a unique solution $k,\, t_*$, and the corresponding motion $\zzz(t), \dot\zzz(t)$ is defined for $t \in \RRR.$

(ii) $r = x$ or $r = -x$. The corresponding solutions are $\zzz(t) = xe^{t}$ or $\zzz(t) = -xe^{-t}$,\, $t \in \RRR$, respectively.

(iii) $0 < x < r$. Here we are looking for a solution in the form $\zzz(t) = k\cosh(t+t_*)$,\, $k > 0$,\, $t \ge -t_*$. The initial conditions are $k\cosh t_* = r,\, k\sinh t_* = x$ or, equivalently, $ke^{t^*} = \dfrac{r+x}{2}$,\, $ke^{-t^*} = \dfrac{r-x}{2}$,\, $k > 0$. These equations have a unique solution $k,\, t_*$, besides $t_* > 0$. Thus, one can take $\ve = t_*$.

(iv) $0 < x < -r$. Similarly to the previous case, we look for the solution in the form $\zzz(t) = -k\cosh(t-t_*),\, t \le t_*$. The initial conditions can be equivalently written as $ke^{t^*} = \dfrac{-r+x}{2}$,\, $ke^{-t^*} = \dfrac{-r-x}{2}$,\, $k > 0$. There is a unique solution $k,\, t_*$ of these equations, with $t_* > 0$. Thus, one can also take $\ve = t_*$.

Let $(z, \check{z}) = (r, 0) \in \Hbs = \RRR^+ \times \{ 0 \}$. The solution in the form \eqref{e1} with the initial conditions $\zzz(0) = r,\, \dot\zzz(0) = 0$ is $\zzz(t) = r(1 + it) e^{-it}$, and so, Im$\,\zzz(t) = r(t\cos t - \sin t) > 0$ for $-\pi/2 < t < 0$. The solution of equations \eqref{e2} with the same initial conditions is $\zzz(t) = r\cosh t,\, t \ge 0$. Thus, here we can take $\ve = \pi/2$.

Let $(z, \check{z}) = (r, 0) \in \Hsb = \RRR^- \times \{ 0 \}$. In a similar way one finds the solution $\zzz(t) = r\cosh t$ for $t \le 0$ and $\zzz(t) = r(1 + it) e^{-it}$ for $0 < t < \pi/2$. Here, again, we have $\ve = \pi/2$.

Finally, if $(z, \check{z}) = (0, 0)$, the solution has the form $\zzz(t) = 0,\, t \in \RRR$.
\end{proof}

The following lemma is, in a sense, inverse to Lemma \ref{l notation}. Its proof is similar to the proof of Lemma \ref{l notation} and therefore is left to the reader.

\begin{lemma}\label{l not2}
Let $(z, \check{z}) \in \HHH$ and $\ve > 0$. Consider a motion $(\zzz(t), \dot\zzz(t))$ defined in a one-sided or two-sided neighborhood of 0 satisfying $(\zzz(0), \dot\zzz(0)) = (z, \check{z})$.

(a) If $\zzz(t) \in \CCC^+$ for $t \in (-\ve,\, 0)$ then $(z, \check{z}) \in \Hb \cup \Hbs$.

(b) If $\zzz(t) \in \CCC^+$ for $t \in (0,\, \ve)$ then $(z, \check{z}) \in \Hb \cup \Hsb$.

(c) If $\zzz(t) \in \RRR$ for $t \in (-\ve,\, 0)$ then $(z, \check{z}) \in \Hs \cup \Hsb \cup \{ (0,0) \}$.

(d) If $\zzz(t) \in \RRR$ for $t \in (0,\, \ve)$ then $(z, \check{z}) \in \Hs \cup \Hbs \cup \{ (0,0) \}$.

\end{lemma}

The main results of this paper are contained in the following theorem.

\begin{theorem}\label{t1}
For any $(z, \check{z}) \in \HHH$ there exists a unique motion $(\zzz(t), \dot\zzz(t)),\, t \in \RRR$, satisfying $(\zzz(0), \dot\zzz(0)) = (z, \check{z})$. There are 5 kinds of motion: (I) billiard motion when $(\zzz(t), \dot\zzz(t)) \in \Hb$ for all $t \in \RRR$; (II) sliding motion, when $(\zzz(t), \dot\zzz(t)) \in \Hs$ for all $t \in \RRR$; (III) transitional motion from billiard to sliding, when $(\zzz(t), \dot\zzz(t))$ lie in $\Hb$ for $t < t_*$, in $\Hs$ for $t > t_*$, and in $\Hbs$ for $t = t_*$; (IV) transitional motion from sliding to billiard, when $(\zzz(t), \dot\zzz(t))$ lie in $\Hs$ for $t < t_*$, in $\Hb$ for $t > t_*$, and in $\Hsb$ for $t = t_*$; (V) $(\zzz(t), \dot\zzz(t)) = (0,0)$ for $t \in \RRR$.

In the case (I) of billiard motion, the following holds.

(a) there are infinitely many instants of hits $\ldots < t_{-2} < t_{-1} < t_0 < t_1 < t_2 < \ldots$ going to infinity in both sides, $\lim_{n\to+\infty} t_n = +\infty$ and $\lim_{n\to-\infty} t_n = -\infty$.

(b) The sequence $\{ r_n = \zzz(t_n) \} \subset \RRR$ is strictly monotone increasing, and takes both negative and positive values, and additionally,
$$
\dfrac{r_{n+1}}{r_n} = 1 + \dfrac{3}{2n} (1 + o(1)) \quad \text{as} \ \ n \to \pm\infty.
$$
It follows that $r_n = \pm |n|^{3/2+o(1)}$ as $n \to \pm\infty$.

(c) One has
$$
{\rm Re}\, \dot\zzz(t_n) = |r_n|\, (1 + o(1)), \quad {\rm Im}\, \dot\zzz(t_n+0) = \frac{3}{2n}\, |r_n|\, (1 + o(1))
$$ $$
\max \{ {\rm Im}\, \zzz(t),\, t_n \le t \le t_{n+1} \} = \frac{9}{16n^2}\, |r_n|\, (1 + o(1)) \quad \text{as} \ \ n \to \pm\infty.
$$

In what follows, without loss of generality assume that $\zzz(t_n) < 0$ for $n \le 0$ and $\zzz(t_n) \ge 0$ for $n \ge 1$, and denote $\tau_n := t_{n+1} - t_n$.

(d) The sequence $\{ \tau_n,\, n \le -1 \}$ is monotone increasing, and the sequence $\{ \tau_n,\, n \ge 1 \}$ is monotone decreasing, with
$$
\tau_n = \dfrac{3}{2|n|}\, (1 + o(1)) \quad \text{as} \ \ n \to \infty.
$$
It follows that $t_n = \pm \dfrac{3}{2} \ln|n|\, (1 + o(1))$ as $n \to \pm\infty.$

(e) There is at most one grazing impact $t_n$, and in the case of grazing, there are no other points of impact in $[0,\, \zzz(t_n)]$, if $\zzz(t_n) > 0$, or in $[\zzz(t_n),\, 0]$, if $\zzz(t_n) < 0$.

In the case (II) there is sliding motion, and $\zzz(t)$ takes one of the forms  $k\sinh(t-t_*)$;\ $ke^{t}$;\ $ke^{-t}$, with $k > 0$ and $t_* \in \RRR$.

In the case (III), for $t < t_*$ there is billiard motion, with the sequences of instants of hits $\ldots < t_{-2} < t_{-1} < t_0$ and of points of impact $\{ r_n = \zzz(t_n),\ n = \ldots, {-2}, {-1}, 0 \} \subset \RRR^-$ satisfying conditions (a)--(d) with corresponding modifications. Grazing does not occur. For $t > t_*$ there is sliding motion in the form $\zzz(t) = k\cosh(t-t_*) \subset \RRR^+$ with $k > 0$.

In the case (IV), for $t < t_*$ there is sliding motion in the form $\zzz(t) = -k\cosh(t-t_*) \subset \RRR^-$ with $k > 0$, while for $t > t_*$ there is billiard motion, with the sequences of instants of hits $t_1 < t_2 < \ldots$ and of points of impact $\{ r_n = \zzz(t_n), \ n = 1,\, 2, \ldots \} \subset \RRR^+$ satisfying conditions (a)--(d) with corresponding modifications. Grazing does not occur.
\end{theorem}

\begin{zam}
The dimension of the phase space $\HHH$ is 4. Theorem \ref{t1} implies that the set of points included in a full billiard motion has full dimension. The set of points included in a billiard motion with grazing has dimension 3, that is, its codimension is 1. Finally, the sets of points included in a mixed motion or in a sliding motion have dimension 2, that is, their codimension equals 2.
\end{zam}

\begin{zam}
Note that our earlier work \cite{KrPl1} contains a part of the results of Theorem \ref{t1}. Namely, only billiard motion (with sliding excluded) with positive time, when the first reflection is from the positive semiaxis, is considered there. It is proved that the sequence $r_n$ is monotone increasing and its growth is sub-exponential and that $\sum \tau_n = \infty$, that is, there is no chattering (infinitely many impacts of a motion in a finite interval of time).
\end{zam}

\section{Basic formulas}

Here we derive some useful formulas concerning billiard motions $(\zzz, \dot\zzz)$, which will be needed later on.

Let $t_n < t_{n+1}$ and assume that $\zzz(t_{n})$ and $\zzz(t_{n+1})$ lie in $\RRR$ and $\zzz(t) \in \CCC^+$ in the interval $t \in (t_n,\, t_{n+1}).$
In this interval we have $\zzz(t+t_n) = (r_n + v_n t) e^{-it}$, where $r_n = \zzz(t_n) \in \RRR$ and $v_n \in \CCC.$

Denote $\dot\zzz(t_n + 0) := x_n + iy_n$, that is, $x_n = \text{Re}\,\dot\zzz(t_n + 0)$ and $y_n = \text{Im}\,\dot\zzz(t_n + 0) \ge 0$. We have
\beq\label{zt}
\zzz(t_n + t) = [r_n + x_n t + i(r_n + y_n)t] e^{-it},
\eeq
and so, Im$\,\zzz(t_n+t) = (r_n + y_n)t \cos t - (r_n + x_n t) \sin t.$ The time interval between two consecutive hits is denoted as $\tau_n := t_{n+1} - t_n$; it is the smallest positive value satisfying
\beq\label{delta}
(r_n + y_n) \tau_n \cos\tau_n - (r_n + x_n\tau_n)\sin\tau_n = 0.
\eeq
Denote $g_n(t) := (r_n + y_n) t \cos t - (r_n + x_n t)\sin t$; the function $g_n$ is positive in $(0,\, \tau_n)$ and equals zero at $t = \tau_n$.

Since $g_n(2\pi) = -2 g_n(\pi)$, the function $g_n$ has at least one zero in $[\pi,\, 2\pi)$. It follows that $0 < \tau_n < 2\pi$.

Further, denote $r_{n+1} := \zzz(t_{n+1}) \in \RRR$,\, $x_{n+1} + iy_{n+1} := \dot\zzz(t_{n+1} -0)^* = \dot\zzz(t_{n+1} +0) \in \CCC^+_0$. Using \eqref{zt} and \eqref{delta} and assuming that $\tau_n \ne \pi$, we obtain
\beqo\label{rnrn}
r_{n+1} = \zzz(t_n + \tau_n) = \text{Re}\,\zzz(t_n + \tau_n) = (r_n + x_n \tau_n) \cos \tau_n + (r_n + y_n) \tau_n \sin \tau_n
\eeqo
$$
= (r_n + y_n) \tau_n \Big( \frac{\cos^2 \tau_n}{\sin\tau_n} + \sin \tau_n \Big),
$$
hence
\beq\label{rn2}
r_{n+1} = \frac{( r_n + y_n)\tau_n}{\sin\tau_n}.
\eeq
Using \eqref{delta}, one derives an alternative formula
\beq\label{rn3}
r_{n+1} = \frac{r_n + x_n \tau_n}{\cos\tau_n},
\eeq
which is valid when $\tau_n \ne \pi/2,\, 3\pi/2$.

Further, using \eqref{delta} and assuming that $\tau_n \ne \pi$, we get
$$
x_{n+1} = \text{Re}\, \dot\zzz(t_{n} +\tau_n -0) = \big[ x_n + (r_n + y_n)\tau_n \big] \cos\tau_n + (y_n - x_n \tau_n) \sin\tau_n
$$ $$
= x_n \cos\tau_n + (r_n + x_n\tau_n)\sin\tau_n + (y_n - x_n \tau_n) \sin\tau_n = x_n \cos\tau_n + (r_n + y_n) \sin\tau_n
$$ $$
=  x_n \cos\tau_n + \frac{(r_n + x_n\tau_n)\sin^2 \tau_n}{\tau_n \cos\tau_n} = \frac{x_n \tau_n + r_n \sin^2 \tau_n}{\tau_n \cos\tau_n}
= \frac{r_n + y_n}{\sin\tau_n} - \frac{r_n \cos\tau_n}{\tau_n},$$
and so,
\beq\label{e-x}
x_{n+1} = \frac{r_n + y_n}{\sin\tau_n} - \frac{r_n \cos\tau_n}{\tau_n}.
\eeq
There is an alternative formula, which is valid when $\tau_n \ne \pi/2,\, 3\pi/2$,
\beq\label{e-x1}
x_{n+1} = \frac{r_n + x_n \tau_n}{\tau_n \cos\tau_n} - \frac{r_n \cos\tau_n}{\tau_n}.
\eeq

Using \eqref{delta} and assuming that $\tau_n \ne \pi$, one obtains
$$
y_{n+1} = -\text{Im}\, \dot\zzz(t_{n} +\tau_n -0) = \big[ x_n + (r_n + y_n)\tau_n \big] \sin\tau_n - (y_n - x_n \tau_n) \cos\tau_n $$ $$
= \frac{(r_n + y_n)\tau_n \cos\tau_n - r_n \sin\tau_n}{\tau_n} + (r_n + y_n)\tau_n \sin\tau_n - (y_n - x_n \tau_n) \cos\tau_n $$ $$
=  (r_n + x_n \tau_n) \cos\tau_n - \frac{r_n \sin\tau_n}{\tau_n} + (r_n + y_n)\tau_n \sin\tau_n $$ $$
= \frac{(r_n + y_n) \tau_n \cos^2 \tau_n}{\sin\tau_n} + (r_n + y_n)\tau_n \sin\tau_n - r_n\, \frac{\sin\tau_n}{\tau_n},
$$
thus
\beq\label{e-y}
y_{n+1}  = (r_n + y_n)\, \frac{\tau_n}{\sin\tau_n} - r_n\, \frac{\sin\tau_n}{\tau_n}.
\eeq
Again, since for $\tau_n = \pi$, this formula is not valid, we propose an alternative formula, which holds for $\tau_n \ne \pi/2,\, 3\pi/2$,
\beq\label{e-y1}
y_{n+1}  = \frac{r_n + x_n \tau_n}{\cos\tau_n} - r_n\, \frac{\sin\tau_n}{\tau_n}.
\eeq

Let us derive some additional useful formulas. Suppose that $r_n \ne 0$ and denote $w_n = a_n + ib_n := v_n/r_n$, so as
\beq\label{anbn}
\zzz(t_n + t) = r_n (1 + w_n t) e^{-it} = r_n (1 + a_n t + ib_n t) e^{-it} \qquad \text{for} \ \ t \in (0,\, \tau_n).
\eeq
One has $a_n = x_n/r_n$ and $b_n = (r_n + y_n)/r_n$. 

The descriptions of dynamics in terms of $(x_n,\, y_n)$ and in terms of $(a_n,\, b_n)$ are equivalent. In proofs of various statements, one or the other description is more convenient.

Using \eqref{rn2}, \eqref{e-x}, and \eqref{e-y}, one obtains the following iterative formulas for $r_n$,\, $a_n$, and $b_n$,
\beq\label{rn4}
r_{n+1} = \frac{b_n \tau_n}{\sin\tau_n}\, r_n,
\eeq
\beq\label{iterA}
a_{n+1} = \frac{1}{\tau_n} - \frac{\cos\tau_n \sin\tau_n}{b_n \tau_n^2},
\eeq
\beq\label{iterB}
b_{n+1} = 2 - \frac{1}{b_n} \Big( \frac{\sin\tau_n}{\tau_n} \Big)^2.
\eeq
Equation \eqref{delta} can be rewritten as
\beq\label{del2}
b_n \tau_n \cos\tau_n - (1 + a_n \tau_n) \sin\tau_n = 0.
\eeq
Using \eqref{del2} and \eqref{iterA}, one derives the following useful formula,
\beq\label{iterAA}
\frac{1}{a_{n+1}} = \tau_n + \frac{1}{a_n + b_n \tan \tau_n}.
\eeq

\section{Proof of Theorem \ref{t1}}

\subsection{Billiard motion}

In this subsection, we assume that $(z, \check{z}) \in \Hb \cup \Hsb$. The latter means that $(z,\check{z}) \not\in (\RRR \times \RRR^+) \cup (\RRR_0^+ \times \{0\})$. Let $t_* \in \RRR$. We are looking for a motion $(\zzz(t), \dot\zzz(t))$, $t \ge t_*$ with the initial condition $(\zzz(t_*), \dot\zzz(t_*)) = (z, \check{z})$.

\begin{lemma}\label{l ini}
There exist a value $t' > t_*$ and a motion $(\zzz(t), \dot\zzz(t))$ on $[t_*,\, t']$  with the initial conditions $(\zzz(t_*), \dot\zzz(t_*)) = (z, \check{z})$ such that

(a) $\zzz(t) \in \CCC^+$ for $t \in (t_*,\, t')$,\, $\zzz(t') \in \RRR$, and $(\zzz(t'), \dot\zzz(t')) \in \Hb \cup \Hbs.$

(b) If $z \in \RRR^+$ then $t' - t_* < \pi$, and if $z = 0$ then $t' - t_* \le \pi$.

(c)    
If $(z, \check{z}) \in \Hsb$ then $\zzz(t') > 0$ and $(\zzz(t'), \dot\zzz(t')) \in \Hb \setminus \Hg$.

(d) If $(\zzz(t'), \dot\zzz(t')) \in \Hg \cup \Hbs$ then $z \not\in \RRR^+_0$.
\end{lemma}

By definition, $t'$ is the smallest instant of hit greater than $t_*.$
Claim (c) means that if there is a passage from sliding to billiard motion at $t_*$, then the consecutive hits at $t_*$ and $t'$ are from the negative and positive semiaxes, respectively, and there is no grazing or passage to sliding at $t'$.
Claim (d) means that a grazing impact or passage from billiard motion to sliding cannot occur after a reflection from the positive semiaxis
    or from the origin.

\begin{proof} (a) By claims (a) and (d) of Lemma \ref{l notation}, there is a motion $(\zzz, \dot\zzz)$ in a right neighborhood of $t_*$ with the given initial conditions such that $z(t_*+t) \in \CCC^+$ for $t \in (0,\, \ve)$. The motion can be uniquely defined by the formula \eqref{e1} and by the initial conditions. Let $t' > t_*$ be the smallest value such that $\zzz(t') \in \RRR$, that is, $t' := \inf \{ t \ge t_* :\, \zzz(t) \in \RRR \}.$ We have $t' < t_* + 2\pi.$
By claim (a) of Lemma \ref{l not2}, $(\zzz(t'), \dot\zzz(t')) \in \Hb \cup \Hbs.$

(b) Denote $z =: r_n \in \RRR^+_0$,\, $\check{z} = \{ x_n - iy_n,\, x_n + iy_n \}$ with $y_n \ge 0$, and $t'-t_* =: \tau_n$. Recall that
$$g(\tau) = g_n(\tau) = (r_n + y_n)\tau \cos\tau - (r_n + x_n \tau) \sin\tau;
$$
$g(\tau) > 0$ for $\tau > 0$ sufficiently small, and $\tau_n$ is the smallest positive zero of $g$. We have $g(\pi) = -\pi (r_n + y_n)$. If $r_n > 0$ then $g(\pi) < 0$, and $\tau_n < \pi.$ If $r_n = 0$ then $g(\pi) \le 0$, and $\tau_n \le \pi.$

(c) Let $(z, \check{z}) \in \Hsb$, that is, $z = \zzz(t_*) \in \RRR^-$ and $\check{z} = \dot\zzz(t_*) = 0$, and
$$\zzz(t_* + t) = z(1 + it)e^{-it}.
$$
The value $\tau_n = t' - t_*$ is the smallest positive value satisfying $\tau_n \cos\tau_n - \sin\tau_n = 0$. One easily sees that $\pi < \tau_n < 3\pi/2$. One has $\zzz(t') = z(\cos\tau_n + \tau_n \sin\tau_n) > 0$, and $\dot\zzz(t_*+t) = zte^{-it}$, hence Im\,$\dot\zzz(t'-0) < 0$, and so, the there are no grazing and no transition to sliding at $t'$.

(d) Let $(\zzz(t'), \dot\zzz(t')) \in \Hg = \RRR \times \RRR^-$ and assume that $z \in \RRR^+_0$. After making identifications $z = \zzz(t_*) =: r_n$,\, $\dot\zzz(t_*+0) =: x_n + iy_n$,\, $\dot\zzz(t'-0) =: x_{n+1} - iy_{n+1}$,\, $t'-t_* =: \tau_n$ we get $x_{n+1} < 0$,\, $y_{n+1} = 0$,\, $r_n \ge 0$. By claim (b) of this lemma we have $0 < \tau_n \le \pi$, hence $\sin\tau_n - \tau_n \cos\tau_n > 0$. First, assume that (i) $\tau_n \ne \pi$. By \eqref{e-y} one has $(r_n + y_n)\, \dfrac{\tau_n}{\sin\tau_n} - r_n\, \dfrac{\sin\tau_n}{\tau_n} = 0$, hence $r_n + y_n = r_n \sin^2 \tau_n/\tau_n^2$, and using \eqref{e-x} one obtains
$$
x_{n+1} = \frac{r_n}{\tau_n^2}\, (\sin\tau_n - \tau_n \cos\tau_n) \ge 0,
$$
in contradiction with our assumption.

Now assume that (ii) $\tau_n = \pi.$ By claim (b) of this lemma, $r_n = z = 0$. Using \eqref{e-y1}, one obtains $y_{n+1} = -\pi x_n$, hence $x_n = 0$,  and using \eqref{e-x1} one gets $x_{n+1} =0$, in contradiction with our assumption.

Let now $(\zzz(t'), \dot\zzz(t')) = (r,0) \in \Hbs = \RRR^+ \times \{0\}$. That is, $r := \zzz(t') > 0$ and $\dot\zzz(t') = 0$. It follows that $\zzz(t' + t) = r(1 + it) e^{-it}$ for $t \in [t_* - t',\, 0]$. Denote $s = t' - t_* > 0$, then $\zzz(t_*) = r(1 - is) e^{is} = r(\cos s + s\sin s) + ir(\sin s - s\cos s)$; thus, $s$ is the smallest positive value satisfying $\sin s - s\cos s = 0$. One concludes that $s > \pi$, hence by the claim (b) of this lemma, $z \not\in \RRR^+_0$.
\end{proof}


\begin{cor}\label{cor sequence}
There is a strictly monotone increasing sequence $\{ t_n \}$ of real values (instants of hits), which may be either infinite, $t_* < t_1 < t_2 < \ldots$, or finite, $t_* < t_1 < t_2 < \ldots < t_m\, (m \ge 1)$, and there is a motion $(\zzz(t), \dot\zzz(t))$ on $[t_*,\, \sup_n t_n)$ with the initial conditions $(\zzz(t_*), \dot\zzz(t_*)) = (z, \check{z})$ satisfying $\zzz(t) \in \CCC^+$ for $t \in (t_*,\, \sup_n t_n) \setminus \{ t_n \}$ and $r_n = \zzz(t_n) \in \RRR,\, \forall n$. One has $(\zzz(t), \dot\zzz(t)) \in \Hb$ for $t \in (t_*,\, \sup_n t_n)$, and $(\zzz(t_m), \dot\zzz(t_m)) \in \Hbs$ if the sequence is finite (the latter means that the billiard motion terminates at $t_m$). 
\end{cor}

\begin{zam}\label{zbs}
If there are finitely many hits, that is, the billiard motion terminates in $\Hbs$ at the instant $t_m$, the motion can be extended to $[t_m,\, +\infty)$ by $$
(\zzz(t), \dot\zzz(t)) = r_m (\cosh(t-t_m), \sinh(t-t_m)), \quad t \ge t_m.
$$
That is, the motion on $[t_*,\, +\infty)$ is mixed: first billiard, then sliding.
\end{zam}

\begin{proof}
The proof of Corollary \ref{cor sequence} is obtained by repeated application of Lemma \ref{l ini}.
\end{proof}

Consider a sequence of successive instants of hits, $t_1 < t_2 < \ldots$, as described in Corollary \ref{cor sequence}. 

\begin{lemma}\label{l2}
(a) The sequence $r_n = \zzz(t_n)$,\, $n \ge 1$ is strictly monotone increasing.

(b) If $r_{1} > 0$ then the sequence of time intervals $\tau_n = t_{n+1} - t_n$,\, $n \ge 1$ is strictly monotone decreasing, and $x_n = \text{\rm Re}\, \dot\zzz(t_n) > 0$ and $y_n = \text{\rm Im}\, \dot\zzz(t_n+0) > 0$ for $n \ge 2$.

(c) If $(\zzz(t_n), \dot\zzz(t_n)) \in \Hg$ (that is, grazing impact takes place) for some $n$ then $\zzz(t_k) < 0$ for $k < n$ and $\zzz(t_k) > 0$ for $k > n$.
It follows that if $r_{n} \ge 0$, no grazing occurs at the subsequent hits $k > n$. Additionally, grazing cannot occur at two consecutive hits, say $n$ and $n+1$.

(d) If $(z, \check{z}) \in \Hsb$, the sequence $\{ t_n \}$ is infinite, and there is no grazing.
\end{lemma}

\begin{proof}
(a) Let $r_{n} > 0$. By claim (b) of Lemma \ref{l ini},\, $0 < \tau_{n} < \pi$, hence $\tau_{n}/\sin\tau_{n} > 1$, and using \eqref{rn2} and taking into account that $y_{n} \ge 0$, one obtains $r_{n+1} > r_{n}$. Applying inductively this argument, one concludes that the sequence $r_{n},\, r_{n+1}, \ldots$ is strictly monotone increasing.

Now suppose that $r_{k} < 0$. Considering the motion with reversed time, according to Remark \ref{z invert}, we see that $-r_{k},\, -r_{k-1}, \ldots, -r_1$ are points of hits for this motion. We have proved in the previous paragraph that this sequence is strictly monotone increasing, and therefore, so does the sequence $r_1, \ldots, r_{k}$.

It remains to note that 0 cannot repeatedly appear in the sequence $r_1, r_2, \ldots$. It follows that the sequence $r_n$ is strictly monotone increasing.

(b) For all $n \ge 1$ we have $0 < \tau_n < \pi$, and by \eqref{delta}, the function
$$
g_{n+1}(\tau) = (r_{n+1} + y_{n+1})\tau \cos\tau - (r_{n+1} + x_{n+1} \tau) \sin\tau
$$
is positive for $0 < \tau < \tau_{n+1}$ and equals zero for $\tau = \tau_{n+1}$. Using \eqref{rn2}, \eqref{e-x}, and \eqref{e-y}, one obtains
$$
g_{n+1}(\tau_n) = \left( \frac{2( r_n + y_n)\tau_n}{\sin\tau_n} - \frac{r_n \sin\tau_n}{\tau_n} \right) \tau_n \cos\tau_n
- \left( \frac{2( r_n + y_n)\tau_n}{\sin\tau_n}  - \frac{r_n \cos\tau_n}{\tau_n} \tau_n \right) \sin\tau_n
$$ $$
= - \frac{2( r_n + y_n)\tau_n}{\sin\tau_n} (\sin\tau_n - \tau_n \cos\tau_n) < 0,
$$
hence $\tau_{n+1} < \tau_n.$

Further, utilizing \eqref{e-x}, for $n \ge 1$ we find
$$
x_{n+1} \ge \frac{r_n}{\sin\tau_n} - \frac{r_n \cos\tau_n}{\tau_n} = r_n\, \frac{2\tau_n - \sin(2\tau_n)}{2\tau_n \sin\tau_n} > 0.
$$

Consider the inclusion
$$(\zzz(t_n), \dot\zzz(t_n)) \in \Hb \cup \Hbs = (\CCC^+ \times \CCC) \cup \big( (\RRR \times (\CCC \setminus \RRR)) /\sim \big) \cup (\RRR \times \RRR^-) \cup \big( \RRR^+ \times \{ 0 \} \big).$$
For $n \ge 2$ we have $r_n = \zzz(t_n) \in \RRR^+$ and $x_n = \text{\rm Re}\, \dot\zzz(t_n) > 0$, hence $(\zzz(t_n), \dot\zzz(t_n)) \in  (\RRR \times (\CCC \setminus \RRR)) /\sim$, and so, $y_n > 0$.

(c) Let now $(\zzz(t_n), \dot\zzz(t_n)) \in \Hg$. By claim (d) of Lemma \ref{l ini},\, $\zzz(t_{n-1}) < 0$. For the motion with reversed time, grazing takes place at the instant $-t_n$ at $-\zzz(t_n)$, and the impact prior to grazing is at the instant $-t_{n+1}$ at the point $-\zzz(t_{n+1})$, hence $-\zzz(t_{n+1}) < 0$.

Let grazing  occur at the hits $n$ and $n+1$; then $y_n = y_{n+1} = 0$, $x_n < 0$, $x_{n+1} < 0$, $r_n \le 0$, and $r_{n+1} \ge 0$. Consider two cases: $\tau_n \ne \pi$ and $\tau_n = \pi$. In the former case by \eqref{e-y} we have
$$
0  =  r_n \Big( \frac{\tau_n}{\sin\tau_n} - \frac{\sin\tau_n}{\tau_n} \Big)
$$
hence $r_n = 0$, and using \eqref{e-x} we obtain the impossible equation $x_{n+1} = 0$. In the latter case, $\tau_n = \pi$, using \eqref{e-y1} we obtain the impossible equation $0 = -(r_n + x_n \tau_n)$.

(d) Let $(z, \check{z}) \in \Hsb$. By claim (c) of Lemma \ref{l ini},\, $r_1 > 0$, and there is no grazing or passage to sliding at $t_1$, that is, $(\zzz(t_1), \dot\zzz(t_1)) \not\in \Hg \cup \Hbs$. Further, for any $n \ge 2$,\, $r_{n-1} \ge \ldots \ge r_1 > 0$, and by claim (d) of Lemma \ref{l ini},\, $(\zzz(t_n), \dot\zzz(t_n)) \not\in \Hg \cup \Hbs$. This means that there is neither grazing nor passage to sliding.
\end{proof}

\begin{cor}\label{cor1}
Let $r_{1} > 0$; then for $n \ge 2$ we have $a_n > 0$,\, $b_n > 1$, and $\tau_n < \pi/2$.
\end{cor}

\begin{proof}
Indeed, the two former inequalities follow from claim (b) of Lemma \ref{l2}. By \eqref{del2}, $\cos\tau_n = \dfrac{(1 + a_{n} \tau_{n})\sin\tau_n}{b_n\tau_n} > 0$, hence $\tau_n < \pi/2$.
\end{proof}

From now on in this section, we assume that the motion under consideration has infinitely many impacts.

\begin{lemma}\label{l1}
Let $r_{1} > 0$.

(a) The time intervals $\tau_n,\, n \ge 1$ strictly monotonically converge to 0:\\ $\tau_n \downarrow 0$ as $n \to \infty$.

(b) $\sum_{n\ge 1} \tau_n = \infty$.
\end{lemma}

\begin{proof}
(a)
For $n \ge 1$ one has $b_{n+1} > 1$, and according to \eqref{iterB}, $b_{n+1} < 2$. Using  \eqref{del2}, for $n \ge 2$ one has
$1 + a_{n} \tau_{n} = {b_n}\, \dfrac{\tau_n}{\tan\tau_n} < 2.$ Thus,
$$
 1 < b_n < 2, \quad 0 < a_n < \frac{1}{\tau_n}, \quad \text{and} \quad \frac{1 + a_n \tau_n}{b_n} < 1 \quad \text{for} \ \, n \ge 2.
$$

By claim (b) of Lemma \ref{l2}, the sequence $\tau_n$ is strictly monotone decreasing, and therefore, converges to a value $c \ge 0$. It remains to prove that $c = 0$.

Assume the contrary, that is, $0 < c < \pi/2$; then by \eqref{del2}, $\dfrac{1 + a_n \tau_n}{b_n} \to \dfrac{c}{\tan c} < 1$ as $n \to \infty$. Since $b_n$ is bounded, there exists a partial limit $\lim_{k\to\infty} b_{n_k} =: \bt$. Using \eqref{iterA} and \eqref{iterB}, one obtains
\beq\label{e5}
\frac{1 + a_{n_k}\tau_{n_k-1}}{b_{n_k}} =
\frac{2 - \dfrac{\cos\tau_{n_k-1} \sin\tau_{n_k-1}}{b_{n_k-1} \tau_{n_k-1}}}{2 - \dfrac{1}{b_{n_k-1}} \Big( \dfrac{\sin\tau_{n_k-1}}{\tau_{n_k-1}} \Big)^2}.
\eeq
Since $a_n$ is bounded, one has
$$
\lim_{k\to\infty} \frac{1 + a_{n_k}\tau_{n_k-1}}{b_{n_k}} =  \lim_{k\to\infty} \frac{1 + a_{n_k}\tau_{n_k}}{b_{n_k}}
= \lim_{k\to\infty} \frac{\tau_{n_k}}{\tan \tau_{n_k}} = \frac{c}{\tan c},
$$
and passing to the limit $k \to \infty$ in \eqref{e5}, we get
$$
\frac{c}{\tan c} =
\frac{2 - \dfrac{\cos c \sin c}{\bt c}}{2 - \dfrac{1}{\bt} \Big( \dfrac{\sin c}{c} \Big)^2},
$$
and therefore, $c = \tan c$. This equation does not have solutions for $c \in (0,\, \pi/2)$. We come to a contradiction.

(b) Here we use the auxiliary function $h(t) = h_{n+1}(t) = \dfrac{t}{\tan t} - \dfrac{1+a_{n+1}t}{b_{n+1}}$. By \eqref{del2},
$$
h(\tau_{n+1}) = \dfrac{\tau_{n+1}}{\tan \tau_{n+1}} - \dfrac{1+a_{n+1}\tau_{n+1}}{b_{n+1}} = 0,
$$
and $h(\tau_n) < 0$, and using \eqref{iterA}, \eqref{iterB}, and the inequality $2 - \dfrac{1}{b_n} \Big( \dfrac{\sin\tau_n}{\tau_n} \Big)^2 = b_{n+1}> 1$ (see Corollary \ref{cor1}), we get
\beq\label{h}
h(\tau_n) = \dfrac{\tau_n}{\tan \tau_n} - \frac{2 - \dfrac{\cos\tau_n \sin\tau_n}{b_n \tau_n}}{2 - \dfrac{1}{b_n} \Big( \dfrac{\sin\tau_n}{\tau_n} \Big)^2}
 = \frac{2}{2 - \dfrac{1}{b_n} \Big( \dfrac{\sin\tau_n}{\tau_n} \Big)^2}\, \Big( \dfrac{\tau_n}{\tan \tau_n} - 1 \Big)
 \eeq
  $$
> 2 \Big( \dfrac{\tau_n}{\tan \tau_n} - 1 \Big) = -\frac{2\tau_n^2}{3}\, (1 + o(1)) \quad \text{as} \ \ n \to \infty.
$$
On the other hand, $h'(t) \le -\dfrac{a_{n+1}}{b_{n+1}} < -\dfrac{a_{n+1}}{2}$ for all $t$. Further, by \eqref{iterAA}, using that $\tau_n < \pi/2$ for $n \ge 2$, one obtains $$\dfrac{1}{a_{n+1}} < \tau_n + \dfrac{1}{a_n}.$$

Assume that $\sum_{n=1}^\infty \tau_n = c < \infty$. Then for $n \ge 2$,\, $\dfrac{1}{a_{n}} < \dfrac{1}{a_{2}} + \sum_{i=2}^{n-1} \tau_i \le \dfrac{1}{a_{2}} + c$, hence $\dfrac{a_n}{2} \ge \dfrac 12\, \Big( \dfrac{1}{a_{2}} + c \Big)^{-1} =: c_1$. Thus, $h'(t) = h_{n+1}'(t) < -c_1$, and
$$
 \tau_n - \tau_{n+1} \le \frac{|h(\tau_{n}) - h(\tau_{n+1})|}{\inf_{t} |h'(t)|} \le \frac{2}{3c_1}\, \tau_n^2\, (1 + o(1))  \quad \text{as} \ \ n \to \infty.
$$
It follows that for $c_2 > 2/(3c_1)$ and for $n$ sufficiently large,
$$
\tau_{n+1} \ge \tau_{n} - c_2 \tau_{n}^2 \ \ \Longrightarrow \ \ \frac{1}{\tau_{n+1}} \le \frac{1}{\tau_{n}} + c_2 (1 + o(1)), \ \ \ n \to \infty,
$$
and so, $\tau_{n} \ge \dfrac{1}{c_2 n} (1 + o(1))$, and $\sum_{n=1}^\infty \tau_{n} = \infty$. We come to a contradiction with our assumption.
\end{proof}

\begin{lemma}\label{l3}
Let $r_1  > 0$; then $a_n \to 1$ and $b_n \to 1$ as $n \to \infty$.
\end{lemma}

\begin{proof}
Using that $\tau_n\to 0$ as $n\to\infty$,\, $b_n > 1$, and taking \eqref{iterB} into account one obtains
$$b_{n+1} = 2-\dfrac{1}{b_n}+\xi_n,$$
where $\xi_n\to 0$ as $n\to \infty$. Hence we have
\beq\label{e7}
b_{n} \ge 2 - \dfrac{1}{b_n} = b_{n+1} - \xi_n.
\eeq

Let $\bt = \lim_{k\to\infty} b_{n_k}$ be the upper partial limit of $b_n$.
It follows from \eqref{e7} that
$$\beta \ge \limsup_{k\to \infty} b_{n_k - 1} \ge \liminf_{k\to \infty} b_{n_k - 1} \ge \lim_{k\to\infty} (b_{n_k}-\xi_{n_k})=\beta. $$
Hence, $\lim_{k\to\infty} b_{n_k - 1}$ exists and coincides with $\bt$. Passing to the limit $k\to\infty$ in the equality
$$b_{n_k} = 2 - \dfrac{1}{b_{n_k - 1}}+ \xi_{n_k - 1},$$
one finds $\bt = 2 - {1}/{\bt}$, whence $\bt = 1$. It follows that $\lim_{n\to\infty} b_n = 1$.

Since by \eqref{del2}, 
$$b_n = \dfrac{(1 + a_n \tau_n) \sin \tau_n}{\tau_n \cos\tau_n},$$ 
making this substitution in \eqref{iterA} and using that by Corollary \ref{cor1}, $a_n > 0$ for $n \ge 2$, after some algebra one obtains
\beq\label{e8}
a_{n+1} = a_n\, \frac{\cos^2 \tau_n}{1 + a_n \tau_n} + \tau_n \Big( \frac{\sin \tau_n}{\tau_n} \Big)^2 <  \frac{a_n}{1 + a_n\tau_n} + \tau_n.
\eeq
Hence we have
\begin{equation}\label{addon}
\Del a_n := a_{n+1} - a_n < \tau_n \Big( 1 - \frac{a_n^2}{1 + a_n \tau_n} \Big) < \tau_n.
\end{equation}
Since $\tau_n \to 0$, for any $0 < \ve < 1$ there exists $n_0 = n_0(\ve)$ such that for all $n \ge n_0$,\, $\tau_n < \ve$, and therefore, for $a > 1 + \ve$ holds 
$$1 - \dfrac{a^2}{1 + a\tau_n}< 1 - \dfrac{a^2}{1 + a\varepsilon}<1 - \dfrac{(1+\varepsilon)^2}{1+\varepsilon+\varepsilon^2} = 
-\dfrac{\varepsilon}{1+\varepsilon+\varepsilon^2}\le 
-\dfrac{\ve}{3}.$$
Here we use the fact that the function  
$$1 - \dfrac{a^2}{1 + a\varepsilon}$$
is strictly monotone decreasing with respect to $a>0$ provided that $\varepsilon$ is positive. It follows that
\beq\label{if}
\text{If} \ \ n \ge n_0 \quad \text{and} \ \ a_n > 1 + \ve \quad \text{then, by Eq.\, \eqref{addon},} \ \  \Del a_n < -\frac{\ve}{3}\, \tau_n.
\eeq

Let us prove that $a_n < 1 + 2\ve$ for $n$ sufficiently large. First, for some $n_1 \ge n_0$ holds $a_{n_1} \le 1 + \ve$; otherwise the sequence $a_n,\, n \ge n_0$ is monotone decreasing with the increments $\Del a_n < -({\ve}/{3})\, \tau_n$, and therefore, tends to $-\infty$.

Second, for all $n > n_1$ the inequality holds $a_n < 1 + 2\ve$. Otherwise let $n_2 > n_1$ be the smallest value for which $a_n \ge 1 + 2\ve$; we have $\Del a_{n_2-1} > 0$, and therefore, by \eqref{if}, $a_{n_2-1} \le 1 + \ve$. On the other hand, $\Del a_{n_2-1} < \tau_{n_2-1} < \ve$, hence $a_{n_2} < 1 + 2\ve$, in contradiction with our assumption.

It follows that $\limsup a_n \le 1$.

Further, from \eqref{e8} one derives
\beqo\label{e9}
\Del a_n = -a_n\, \frac{\sin^2 \tau_n}{1 + a_n \tau_n} + \tau_n \left[ \Big( \frac{\sin \tau_n}{\tau_n} \Big)^2 - \frac{a_n^2}{1 + a_n \tau_n} \right]
\eeqo
\beq\label{e10}
> \tau_n \left[ -a_n\, {\tau_n} + \Big( \frac{\sin \tau_n}{\tau_n} \Big)^2 - a_n^2 \right].
\eeq
Let us show that for any $0 < \ve < 1$ there exist infinitely many values of $n$ for which $a_n > 1 - \ve$. Indeed, otherwise all $a_n$ for $n$ sufficiently large lie in $[0,\, 1-\ve]$, and the sum over $n$ of the right hand sides in \eqref{e10} is greater than $\ve \sum_{n=1}^\infty \tau_n (1 + o(1))$, and therefore, diverges to $+\infty$. It follows that $a_n \to +\infty$, which is impossible.

By \eqref{e10}, since both sequences $a_n$ and $\tau_n$ are bounded, there exists a constant $c > 0$ such that  $\Del a_n \ge -c \tau_n$.
Fix $0 < \ve < 1$. Since $\tau_n \to 0$, by \eqref{e10},  there exists $m$ such that for all $n \ge m$, the inequality $0 \le a_n \le 1 - \ve$ implies $\Del a_n \ge ({\ve}/{2})\, \tau_n > 0$. Choose a subsequence $m < n_1 < n_2 < \ldots n_k < \ldots$ such that $a_{n_k} > 1 - \ve$ for all $k$.

Let $a_{s_k}$ be the smallest value among $\{a_{n_k+1},\, a_{n_k+2}, \ldots, a_{n_{k+1}}\}$. Take $k$ such that $n_k \ge m$. If $a_{s_k} \le 1 - \ve$ then $a_{s_k-1} > 1 - \ve$. Indeed, if $s_k = n_k+1$, this is obvious, and if $s_k \ge n_k+2$ then $\Del a_{s_k-1} \le 0$, hence $a_{s_k-1} > 1 - \ve$. We have
$$a_{s_k} = a_{s_k-1} + \Del a_{s_k-1} > 1 - \ve - c\tau_{s_k-1}.$$
Since $\tau_{s_k-1}$ converges to zero, we conclude that the lower partial limit of $a_n$ is $\ge 1 - \ve$, and since $\ve$ can be made arbitrarily small, $\liminf a_n \ge 1$. Lemma \ref{l3} is proved.
\end{proof}

\begin{lemma}\label{l4}
Let $r_1 > 0$. We have $\tau_n = \dfrac{3}{2n} (1 + o(1))$,\, $b_n = 1 + \dfrac{3}{2n} (1 + o(1))$, and $\dfrac{r_{n+1}}{r_n} = 1 + \dfrac{3}{2n} (1 + o(1))$ as $n \to \infty$.
\end{lemma}

\begin{proof}
Using that $a_n = 1 + o(1)$ and $\tau_n \to 0$ as $n \to \infty$, by formula \eqref{del2} we find
$$
\bt_n := b_n - 1 = (1 + a_n \tau_n)\,  \frac{\tan\tau_n}{\tau_n} - 1 = \tau_n + o(\tau_n),
$$
Substituting this relation in \eqref{iterB}, one obtains
$$
\bt_{n+1}  = 1 - \frac{1}{1+\bt_n}\, \Big( \frac{\sin\tau_n}{\tau_n} \Big)^2=
1 - \big[ 1 - \bt_n + \bt_n^2 + o(\bt_n^2) \big] \Big[ 1 - \frac{\tau_n^2}{3} + o(\tau_n^2) \Big]
$$
$$
= \bt_n - \bt_n^2 + o(\bt_n^2) + \frac{1}{3} \tau_n^2 + o(\tau_n^2)\ =\ \bt_n - \frac{2}{3} \tau_n^2 + o(\tau_n^2).
$$
Thus,
$$
\frac{1}{\bt_{n+1}} = \frac{1}{\bt_{n}} \Big( 1 - \frac{2}{3} \frac{\tau_n^2}{\bt_n} + \frac{o(\tau_n^2)}{\bt_n} \Big)^{-1}
= \frac{1}{\bt_{n}} + \frac{2}{3} + o(1), \quad n \to \infty.
$$
It follows that $\dfrac{1}{\bt_{n}} = \dfrac 23\, n\, (1 + o(1))$, hence
$$
\bt_n = \frac{3}{2n}\, (1 + o(1)) \quad \text{and} \quad \tau_n = \frac{3}{2n}\, (1 + o(1)), \quad n \to \infty.
$$

Finally, by \eqref{rn4} we have $\dfrac{r_{n+1}}{r_n} = b_n\, \dfrac{\tau_n}{\sin\tau_n} = 1 + \dfrac{3}{2n}\, (1 + o(1))$.
Lemma \ref{l4} is proved.
\end{proof}

\begin{lemma}\label{l Im}
Let $r_1 > 0$. We have $\max \{ {\rm Im}\, \zzz(t),\, t_n \le t \le t_{n+1} \} = \dfrac{9}{16n^2}\, r_n\, (1 + o(1))$ as $n \to \infty$.
\end{lemma}

\begin{proof}
At the maximum of ${\rm Im}\, \zzz$ we have ${\rm Im}\, \dot\zzz(t_n + s) = 0$. Using formula \eqref{anbn}, one comes to the equation $b_n -1-a_n s = (a_n +b_n s)\tan s$, and using the asymptotics of $a_n$,\, $b_n$, and $\tau_n$ obtained in Lemmas \ref{l3} and \ref{l4} (recall that $s < \tau_n$), one obtains 
$$s = \frac{3}{4n}\, (1 + o(1))$$ 
as $n \to \infty$. Using again formula \eqref{anbn}, one gets
$$
\max \{ {\rm Im}\, \zzz(t),\, t_n \le t \le t_{n+1} \} = {\rm Im}\, \zzz(t_n + s)
$$ $$
= r_n \cos s [b_n s - (1 + a_n s)\tan s] =  \frac{9}{16n^2}\, r_n\, (1 + o(1)).
$$
\end{proof}

In the following Lemma \ref {l5}, $r_1$ may be both positive and negative.

\begin{lemma}\label{l5}
There is $n_0$ such that $r_{n_0} = \zzz(t_{n_0}) > 0$; that is, a certain reflection is from the positive semi-axis.
\end{lemma}

\begin{proof}
If the sequence $\{ t_n \}$ is finite and terminates at $t_m$ with $(\zzz(t_m), \dot\zzz(t_m)) \in \Hbs$, there is nothing to prove, since $r_m \in \RRR^+$. Suppose now that the sequence is infinite.

Assume the contrary, that is, all hits in the infinite sequence $\{ t_n \}$ are at the negative semiaxis: $r_n < 0$ for all $n$. By claim (a) of Lemma \ref{l2}, the sequence $r_{n} = \zzz(t_n)$ is strictly monotone increasing. Considering the reversed motion and applying claims (b) and (c) of Lemma \ref{l2} and the latter equation in Corollary \ref{cor1}, we conclude that the sequence of time intervals $\tau_n$ is strictly monotone increasing and that $x_n > 0$,\, $y_n > 0$, and $0 < \tau_n < \pi/2$.  It follows that $0 < r_{n+1}/r_n < 1$, and by formula \eqref{rn2},
\beq\label{frac}
0 < \frac{r_n + y_n}{r_n}\, \frac{\tau_n}{\sin\tau_n} < 1.
\eeq
It follows that
$$r_n + y_n < 0.$$

Denote $\tau_* := \lim_{n\to\infty} \tau_n \in (0,\, \pi/2]$,\, $c_n := \Big( \dfrac{\sin\tau_n}{\tau_n} \Big)^2$, and $c_* := \Big( \dfrac{\sin\tau_*}{\tau_*} \Big)^2$; we have $c_n \downarrow c_*$,\, $2/\pi < c_n < 1$, and $2/\pi \le c_* < 1$.

Recall the notation
$$
a_n = \frac{x_n}{r_n} \quad \text{and} \quad b_n = \frac{r_n + y_n}{r_n}.
$$
One has $a_n < 0$ and $0 < b_n < 1$, and by \eqref{iterB} one has $b_{n+1} = 2 - \dfrac{c_n}{b_n}.$

Denote by $b_n^*$ the smallest value of $b$ solving the equation $b + \dfrac{c_n}{b} = 2$. Let us show that $b_n \le b_n^*$ for all $n$. Indeed, assume the contrary: for some $m$ holds $b_m > b_m^*$. Since $b_m \in (b_m^*,\, 1)$, we have $b_m + \dfrac{c_m}{b_m} < 2$, and so,
$$
b_{m+1}^* < b_m^* < b_m < 2 - \dfrac{c_m}{b_m} = b_{m+1},
$$
and repeating this argument, by induction one obtains that $b_{n}^* < b_{n}$ and $b_{n} < b_{n+1}$ for all $n \ge m$. Thus, the sequence $b_n,\, n \ge m$, is strictly monotone increasing, and since the sequence $c_n$ is strictly monotone decreasing and $b_{n+1} = 2 - \dfrac{c_n}{b_n} < 1$, one concludes that $b_n < c_n$, and
$$
b_{n+1} - b_n = 2 - \Big( b_n + \frac{c_n}{b_n} \Big) \ge 2 - \Big( b_n + \frac{c_m}{b_n} \Big) \ge 2 - \Big( b_m + \frac{c_m}{b_m} \Big) = b_{m+1} - b_m;
$$
the penultimate inequality is true since $b_m \le b_n$ and $b_m b_n < b_m < c_m$. Thus, the increments $b_{n+1} - b_n,\, n \ge m$, are greater than a positive constant, which contradicts the inequality $b_n < 1$,\, $\forall n$.

Now using that $b_n \le b_n^*$ prove that the sequence $b_n$ is monotone decreasing. Indeed, we have $b_n + \dfrac{c_n}{b_n} \ge 2$, hence
$$
b_{n+1} = 2 - \frac{c_n}{b_n} \le b_n.
$$
Denote by $b_*$ the limit of $b_n$; we have $b_* + \dfrac{c_*}{b_*} = 2$ and $0 < b_* < 1$.

By \eqref{del2}, the sequence $a_n$ converges to
$$a_* := \frac{b_*\cos\tau_*}{\sin\tau_*} - \frac{1}{\tau_*},
$$
and therefore, is bounded.

Recall that the function $h_n$ is defined as
$$
h_n(\tau) = \frac{\tau\cos\tau}{\sin\tau} - \frac{1 + a_n \tau}{b_n}.
$$
The sequence $\sup_{[\tau_n, \tau_{n+1}]} |h_n'(t)|$ is bounded, indeed,
$$
\sup_{[\tau_n, \tau_{n+1}]} |h_n'(t)| \le \sup_{[0, \pi/2]} |h_n'(t)| \le \frac{\pi}{2} + \sup_n \Big| \frac{a_n}{b_n} \Big| \le \frac{\pi}{2} + \frac{\sup_n |a_n|}{b_*} < \infty.
$$
Further, one has $h_n(\tau_n) = 0$ and $\tau_{n+1} - \tau_n \to 0$ as $n \to \infty$, hence
$$
|h_{n+1}(\tau_{n+1}) - h_{n+1}(\tau_n)| \le \sup_{[\tau_n, \tau_{n+1}]} |h_{n+1}'(t)|\, (\tau_{n+1} - \tau_n) \to 0.
$$
On the other hand, using \eqref{h}, one finds that the increments
$$
|h_{n+1}(\tau_n) - h_{n+1}(\tau_{n+1})| = |h_{n+1}(\tau_n)|
 = \frac{2}{2 - \dfrac{c_n}{b_n}}\, \Big(1 - \dfrac{\tau_n}{\tan \tau_n} \Big) > 1 - \dfrac{\tau_1}{\tan \tau_1},
 $$
are greater than a positive constant, and therefore, do not converge to 0. This contradiction proves the lemma.
\end{proof}

Let us summarize the results obtained in this subsection.

We look for a motion $(\zzz(t), \dot\zzz(t))$ with the initial condition $(\zzz(t_*), \dot\zzz(t_*)) = (z, \check{z}) \in \Hb \cup \Hsb$.

\begin{quote}

$\bullet$ If $(z, \check{z}) \in \Hb$ then either there is a billiard motion defined on $[t_*,\, +\infty)$ with infinitely many hits, or there is a billiard motion defined on $[t_*,\, t_{b\to s}]$ with finitely many hits that terminate in $\Hbs$ at the final instant $t_{b\to s}$.

$\bullet$ If $(z, \check{z}) \in \Hsb$ then there is a billiard motion defined on $[t_*,\, +\infty)$ with infinitely many hits.

$\bullet$ If $(z, \check{z}) \in \Hb$, grazing may occur at most once (say, at the instant $t_n$), and there are no other points of impact on $[0,\, \zzz(t_n)]$ or on $[\zzz(t_n),\, 0]$. If $(z, \check{z}) \in \Hsb$, grazing does not occur.

$\bullet$ The sequence $r_n =\zzz(t_n)$ is strictly monotone increasing. Assume that $r_n > 0$ for $n \ge n_0$; then the sequence of time intervals $\tau_n = t_{n+1} - t_n$ is strictly monotone decreasing, and additionally, $\tau_{n_0} < \pi$ and $\tau_n < \pi/2$ for $n \ge n_0 + 1$.

$\bullet$ If there are infinitely many hits, then
$$
\dfrac{r_{n+1}}{r_n} = 1 + \dfrac{3}{2n} (1 + o(1)), \ \tau_n = \dfrac{3}{2n} (1 + o(1)), \ a_n = 1 + o(1), \ b_n = 1 + \dfrac{3}{2n} (1 + o(1)),
$$
and \ $\max \{ {\rm Im}\, \zzz(t),\, t_n \le t \le t_{n+1} \} = \dfrac{9}{16n^2}\, r_n\, (1 + o(1)) \ \ \text{as} \ \ n \to \infty.$

\end{quote}

\subsection{Sliding and mixed motions, and uniqueness of solutions}

Now consider the case when $(z, \check{z}) \in \Hs \cup \Hbs \cup \{ (0,0) \}$. We are looking for a motion $(\zzz(t), \dot\zzz(t))$, $t \ge t_*$ with the condition $(\zzz(t_*), \dot\zzz(t_*)) = (z, \check{z})$.

Let $(z, \check{z}) = (r,x) \in \Hs =\RRR \times \RRR^+$. According to the proof of Lemma \ref{l notation}, there may be several cases. If (i) $|r| < x$ then a solution is given by $\zzz(t) = k\sinh(t-\tilde t)$ with appropriate $k > 0$ and $\tilde t$. If (ii) $r=x > 0$ or $r=-x< 0$ then a solution is $\zzz(t) = xe^{t-t_*}$ or $\zzz(t) = -xe^{-(t-t_*)}$, respectively. If (iii) $r > x > 0$ then a solution is $\zzz(t) = k\cosh(t-\tilde t)$ with appropriate $k > 0$ and $\tilde t < t_*$. Finally, if (iv) $r < -x < 0$ then a solution is given by $\zzz(t) = -k\cosh(t-\tilde t)$, $t_* \le t \le \tilde t$, with appropriate $k > 0$ and $\tilde t > t_*$. The corresponding motion $(\zzz(t), \dot\zzz(t))$ terminates at the instant $t = \tilde t$ at the point $(-k, 0) \in \Hsb$, and can be extended to $[\tilde t,\, +\infty)$, according to the previous subsection: for $t \ge \tilde t$ the motion will be billiard.

Let $(z, \check{z}) = (r,0) \in \Hbs$,\, $r > 0$. Again, according to the proof of Lemma \ref{l notation}, there is a sliding solution given by $\zzz(t) = r[1 + i(t-t_*)] e^{-i(t-t_*)}$, $t \ge t_*$.

Finally, a solution with the initial conditions $(0,0)$ is given by $\zzz(t) = 0$.

Thus, we have proved that there always exists a motion with $t \ge t_*$, whatever the initial conditions at $t=t_*$. Taking the motion with reversed time, according to Remark \ref{z invert}, one concludes that there also exists a motion with $t \le t_*$ with the condition $(z, \check{z})$ at $t=t_*$. It follows that for all $(z, \check{z}) \in \HHH$ there exists a full motion $(\zzz(t), \dot\zzz(t)),\, t \in \RRR$ satisfying $(\zzz(t_*), \dot\zzz(t_*)) = (z, \check{z})$, and there are 5 kinds of motion:
\vspace{1mm}

(I) Billiard motion; all points $(\zzz(t), \dot\zzz(t))$ lie in $\Hb$.
\vspace{1mm}

(II) sliding motion; all points $(\zzz(t), \dot\zzz(t))$ lie in $\Hs$.
\vspace{1mm}

(III) Billiard followed by sliding; there is $t_*$ such that all points $(\zzz(t), \dot\zzz(t)),\, t<t_*$, lie in $\Hb$, all points $(\zzz(t), \dot\zzz(t)),\, t>t_*$, lie in $\Hs$, and $(\zzz(t_*), \dot\zzz(t_*)) \in\Hbs$.
\vspace{1mm}

(IV) Sliding followed by billiard; here, vice versa, points $(\zzz(t), \dot\zzz(t))$ with $t<t_*$, lie in $\Hs$, points $(\zzz(t), \dot\zzz(t))$ with $t>t_*$, lie in $\Hb$, and $(\zzz(t_*), \dot\zzz(t_*)) \in\Hsb$.
\vspace{1mm}

(V) the ball is resting at the origin; $(\zzz(t), \dot\zzz(t)) = (0,0),\, \forall t$.
\vspace{1mm}

All corresponding trajectories $\{ (\zzz(t), \dot\zzz(t)) : t \in \RRR \}$ either are disjoint, or coincide, and their union is the phase space $\HHH$.

These motions are explicitly determined; let us call them {\it canonical} ones. Restrictions of canonical motions to time intervals will also be called canonical.

\begin{lemma}
All motions are canonical. It follows that the motion satisfying $(\zzz(t_*), \dot\zzz(t_*)) = (z, \check{z})$ for $(z, \check{z}) \in \HHH$ and $t_* \in \RRR$ is unique.
\end{lemma}

\begin{proof}
First note that each canonical billiard solution either does not contain grazing points or contains exactly one grazing point. Further, a solution is locally uniquely defined by a condition in $(\CCC^+ \times \CCC) \cup \big( (\RRR \times (\CCC \setminus \RRR)) /\sim \big) = \Hb \setminus \Hg$, and is canonical. It follows that for $(z, \check{z}) \in \Hb \setminus \Hg$, the condition $(\zzz(t_*), \dot\zzz(t_*)) = (z, \check{z})$ uniquely defines a canonical billiard motion defined either on $\RRR$, or on a semi-infinite open interval containing $t_*$, taking values in $\Hb \setminus \Hg$ (that is, without grazing). In the latter case, the state of the system at the endpoint of the interval belongs to $\Hg \cup \Hbs \cup \Hsb$.

Further, clearly, if a motion is sliding, that is, $\zzz \in \RRR$, on a time interval, then it is canonical on this interval.

Consider a motion $(\zzz(t), \dot\zzz(t))$, $t \in \RRR$. If all $\zzz(t)$ lies in $\RRR$ then the motion is canonical: it is either sliding of kind (II) or a resting point of kind (V). If, otherwise, a point $\zzz(t^0)$ lies in $\CCC^+$, then either the motion is canonical of kind (I) without grazing, or there is a semi-infinite open interval, say, $(-\infty,\, t')$ containing $t^0$ such that the motion is billiard canonical without grazing for $t < t'$ and $(\zzz(t'), \dot\zzz(t') \in \Hg \cup \Hbs$.

In the latter case, if all $\zzz(t),\, t>t'$ lie in $\RRR$ then the motion is sliding in $[t',\, +\infty)$, and by Lemma \ref{l not2}, $(\zzz(t'), \dot\zzz(t')) \in \Hbs$. Thus, the motion is canonical of kind (III). If, otherwise, a point $\zzz(t^1),\, t^1 > t'$ lies in $\CCC^+$ then, again, there is a semi-infinite open interval $(t'',\, +\infty)$ containing $t^1$ such that the motion $(\zzz(t), \dot\zzz(t))$ is billiard without grazing for $t>t''$ and $\zzz(t'') \in \Hg \cup \Hsb$.

No point in $(t',\, t'')$ can lie in $\CCC^+$, since otherwise one of the points $t'$ and $t''$ is included in an interval where the motion takes values in the set $\Hb \setminus \Hg$ (billiard without grazing), in contradiction with the fact that the states of the system at $t'$ and $t''$ do not belong to this set. It follows that the motion is sliding on $[t',\, t'']$; additionally, it takes values in $\Hg \cup \Hbs \cup \Hsb$ at the endpoints $t'$ and $t''$. This may happen only in the case when the interval degenerates to a point, $t' = t''$, and the system is in $\Hg$ at the instant $t'$, that is, the motion is canonical of kind (I).
\end{proof}

\section*{Acknowledgements}
The work of the first author was supported by Gda\'{n}sk University of Technology by the DEC 14/2021/IDUB/I.1 grant under the Nobelium - 'Excellence Initiative - Research University' program. The second author was supported by the Center for Research and Development in Mathematics and Applications (CIDMA) through the Portuguese Foundation for Science and Technology (FCT), within the projects UIDB/04106/2020, UIDP/04106/2020, and CoSysM3, ref. 2022.03091.PTDC. We thank the anonymous referee for their valuable comments. 

Declarations of interest: none.

Declaration of generative AI in scientific writing: no AI or AI-assisted technologies were used in the writing process.


\begin{thebibliography}{99}

\bibitem{BreathingCircle}
Claudio Bonanno and Stefano Mar\`{o}. {\it Chaotic motion in the breathing circle billiard}. Ann. Henri Poincar\'{e} {\bf 23}, 255-291 (2022).

\bibitem{rotSquare}
S. Borgan, R.C. Johnson {\it Rotating square billiard}. Phys. Lett. A {\bf 262}, 427-433 (1999).

\bibitem{BrFK}
F. Brock, V. Ferone and B. Kawohl.\, \textit{A symmetry problem in the calculus of variations}.\, Calc. Var. {\bf 4}, 593-599 (1996).

\bibitem{Burdzy}
K. Burdzy, M. Duarte, C.-E. Gauthier, C R. Graham, J. San Martin. {\it Fermi acceleration in rotating drums.} J. Math. Phys. {\bf 63}, 062706 (2022).

\bibitem{BFK}
G. Buttazzo, V. Ferone, B. Kawohl.\,
\textit{Minimum problems over sets of concave functions and related questions}.\, Math. Nachr. {\bf 173}, 71--89 (1995).

\bibitem{BK}
G. Buttazzo, B. Kawohl.\, \textit{On Newton's problem of minimal resistance}.\, Math. Intell. {\bf 15}, 7--12 (1993).

\bibitem{BG97}
G. Buttazzo, P. Guasoni.\, {\it Shape optimization problems over classes of convex domains}.\, J. Convex Anal. {\bf 4}, No.2, 343-351 (1997).

\bibitem{Car}
R. E. de Carvalho, F.\,C. de Souza, and E.\,D. Leonel. {\it Fermi acceleration on the annular billiard: a simplified version.} J. Phys. A {\bf 39}, 3561-73 (2006).

\bibitem{rotOval}
D. R. da Costa, D. F. M. Oliveira, and E. D. Leonel. {\it Dynamical and statistical properties of a rotating oval billiard}. Commun. Nonlinear. Sci. Numer. Simulat. {\bf 19}, 1926-1934 (2014).

\bibitem{Dol}
D. Dolgopyat. {\it Fermi acceleration}, in Geometric and Probabilistic Structures in Dynamics, Contemporary Mathematics Vol. 469 (American Mathematical Society, Providence, RI, 2008), pp. 149-166.

\bibitem{rot1}
D.\,B. Fairlie and D.\,K. Siegwart. {\it Classical billiards in a rotating boundary}. J. Phys. A. {\bf 21}, 1157-1165  (1988).

\bibitem{rot2}
H. Frisk and R. Arvieui. {\it Rotating billiards}. J. Phys. A {\bf 22}, 1765-1778  (1989).

\bibitem{Gel2008}
V. Gelfreich and D. Turaev. {\it Fermi acceleration in non-autonomous billiards}. J. Phys. A {\bf 41} (2008), 212003.

\bibitem{Gel1}
V. Gelfreich, V. Rom-Kedar, and D. Turaev. {\it Fermi acceleration and adiabatic invariants for non-autonomous billiards.} Chaos {\bf 22}, 033116 (2012).

\bibitem{Gel2}
V. Gelfreich, V. Rom-Kedar, and D. Turaev. {\it Oscillating mushrooms: adiabatic theory for a non-ergodic system.} J. Phys. A {\bf 47}, 395101 (2014).

\bibitem{INei}
A\,.P. Itin, A.\,I. Neishtadt, and A.\,A. Vasiliev. {\it Resonant phenomena in slowly perturbed rectangular billiards}. Phys. Lett. A {\bf 291}, 133-138 (2001).

\bibitem{Kr}
S. Kryzhevich.\, {\it Motion of a Rough Disc in Newtonian Aerodynamics}.\, In Optimization in the Natural Sciences. EmC-ONS 2014. Communications in Computer and Information Science, vol 499. Springer, 3-19 (2015).

\bibitem{KrPl1}
S. Kryzhevich and A. Plakhov.\, {\it Billiard in a rotating half-plane}.\, J. Dynam. Control Syst., in press.

\bibitem{LO}
T. Lachand-Robert and E. Oudet.\, \textit{Minimizing within convex bodies using a convex hull method}.\, SIAM J. Optim. {\bf 16}, 368-379 (2006).

\bibitem{LP1}
T. Lachand-Robert, M.~A. Peletier.\,
\textit{Newton's problem of the body of minimal resistance in the class of convex developable functions}.\, Math. Nachr. {\bf 226}, 153-176 (2001).

\bibitem{Los}
A. Loskutov, A.\,B. Ryabov, and L.\,G Akinshin. {\it Properties of some chaotic billiards with time-dependent boundaries.} J. Phys. A {\bf 33}, 797-86 (2000).


\bibitem{N}
I. Newton. {\it Philosophiae naturalis principia mathematica}. (London: Streater) 1687.

\bibitem{ARMA} A. Plakhov. {\it Billiards and two-dimensional problems of optimal resistance}. Arch. Ration. Mech. Anal. {\bf 194}, 349-382 (2009). 

\bibitem{canad}  A. Plakhov. {\it Optimal roughening of convex bodies}. Canad. J. Math. {\bf 64}, 1058-1074 (2012).

  \bibitem{PT}  A. Plakhov and D. Torres.\, {\it Newton's aerodynamic problem in media of chaotically moving particles}. Sbornik: Math. {\bf 196}, 885-933 (2005).

\bibitem{SIREV}
A. Plakhov. {\it Problems of minimal resistance and the Kakeya problem.} SIAM Review {\bf 57}, 421-434 (2015).

 \bibitem{PTG}
A. Plakhov, T. Tchemisova and P. Gouveia. {\it Spinning rough disk moving in a rarefied medium}. Proc. R. Soc. Lond. A. {\bf 466}, 2033-2055 (2010).

\bibitem{OMT}
A. Plakhov and T. Tchemisova. {\it Problems of optimal transportation on the circle and their mechanical applications}. J. Diff. Eqs. {\bf 262}, 2449-2492 (2017).

\end{thebibliography}
\end{document}